\documentclass[oneside]{amsart}
\usepackage[utf8]{inputenc}

\usepackage[dvipsnames]{xcolor}
\usepackage[margin=1in]{geometry}
\usepackage{fontsize}\changefontsize{12pt}

\usepackage{amsmath}
\usepackage{amssymb, verbatim}
\usepackage{tikz}
\usepackage{comment}

\usepackage[backend=biber, style=alphabetic, sorting=nty]{biblatex}
\usepackage{amsmath}
\usepackage{amsfonts}
\usepackage{amsthm}
\usepackage{amscd}
\usepackage{amstext}
\makeatletter
\let\@afterindenttrue\@afterindentfalse
\@afterindentfalse
\makeatother
\usepackage{soul}

\usepackage{mathrsfs}
\usepackage{enumitem}
\usepackage[hyphens]{xurl}
\usepackage{hyperref}
\hypersetup{
	colorlinks=true,
	linkcolor=Fuchsia,
    citecolor=Maroon,
	filecolor=Magenta,      
	urlcolor=NavyBlue,
}
\usepackage[hyphenbreaks]{breakurl}
\usepackage{graphicx}
\usepackage{caption}
\usepackage{subcaption}
\graphicspath{ {./images/} }
\usepackage{tikz-cd}
\usepackage{cleveref}
\usepackage[normalem]{ulem}

\newtheorem{thm}{Theorem}[section]
\newtheorem{prop}[thm]{Proposition}

\newtheorem{que}{Question}
\theoremstyle{definition}
\newtheorem{defn}[thm]{Definition}
\newtheorem{lem}[thm]{Lemma}

\theoremstyle{definition}
\newtheorem{ex}[thm]{Example}

\newcommand{\dual}[1]{#1^{\text{\textexclamdown}}}
\newcommand{\cobar}[1]{\Omega \dual #1}
  
\newcommand{\C}{\mathbb{C}}

\newcommand{\Spec}{\text{Spec}\,}
\newcommand{\BSpec}{{\textbf{\text{Spec}\,}}}

\newcommand{\A}{\mathcal{A}}
\newcommand{\bC}{\mathbb{C}}
\newcommand{\bL}{\mathbb{L}}
\newcommand{\bT}{\mathbb{T}}

\newcommand{\dd}{d_{dR}}
\newcommand{\Hom}{\text{Hom}}
\newcommand{\RHom}{\text{RHom}}
\newcommand{\Ext}{\text{Ext}}
\newcommand{\End}{\text{End}}

\newcommand{\Sym}{\text{Sym}}

\newcommand{\la}{\langle}
\newcommand{\ra}{\rangle}
\newcommand{\mnc}{\ensuremath{\boldsymbol{\mathcal{M}}^d(\C^n)}}
\newcommand{\tmc}[1]{\ensuremath{\mathcal{M}^d(\C^{#1})}}
\newcommand{\mdc}[1]{\ensuremath{\boldsymbol{\mathcal{M}}^d(\C^{#1})}}
\newcommand{\Der}{\mathbb{D}er}

\newcommand{\mb}[1]{\mathbb #1}
\newcommand{\mc}[1]{\mathcal #1}
\newcommand{\mf}[1]{\mathfrak #1}
\newcommand{\ms}[1]{\mathscr #1}
\newcommand{\drep}[1]{\textbf{Rep}_{#1}(A)}
\newcommand{\tr}{\text{Tr}}

\title{Moduli of sheaves on fourfolds as derived Lagrangian intersections}
\author{Nachiketa Adhikari}
\address{Department of Mathematics, University of Illinois at Urbana-Champaign, Urbana IL 61801, USA}
\email{na17@illinois.edu}

\author{Yun Shi}
\address{Department of Mathematics, Brandeis University, Waltham, MA 02453, USA}
\email{yunshi@brandeis.edu}

\begin{document}

\maketitle
\begin{abstract}
    We show that any $(-2)$-shifted symplectic derived scheme $\textbf{X}$ (of finite type over an algebraically closed field of characteristic zero) is locally equivalent to the derived intersection of two Lagrangian morphisms to a $(-1)$-shifted symplectic derived scheme which is the $(-1)$-shifted cotangent stack of a smooth classical scheme. This leads to the possibility of the following viewpoint that is, at least to us, new: any $n$-shifted symplectic derived scheme can be obtained, locally, by repeated derived Lagrangian intersections in a smooth classical scheme.
    
    We also give a separate proof of our main result in the case where the local Darboux atlas cdga for $\textbf{X}$ has an even number of generators in degree $(-1)$; in this case we strengthen the result by showing that $\textbf{X}$ is in fact locally equivalent to the derived critical locus of a shifted function, which we've been told is a folklore result in the field. We indicate the implications of this for derived moduli stacks of sheaves on Calabi-Yau fourfolds by spelling out the case when the fourfold is $\mathbb{C}^4$.
\end{abstract}

\section{Introduction}
It is well-known (see, for example, \cite[\S 2.2]{bbs}, \cite[\S 2]{ricolfisavvas}) that the moduli stack $\tmc 3$ of coherent sheaves of length $d$ on $\mb C^3$ can be written as the critical locus of a regular function on an ambient stack i.e. let $R := End_{\mb C} (\mb C^d)^3$, the affine space of triples of $d \times d$ matrices; the function $f: R \to \mb C$ defined by
\[ f(A, B, C) = \tr (A[B, C]) \]
descends to the quotient stack $\mf M = [R/GL_d]$ of $R$ by the conjugation action of $GL_d$, and $\tmc 3$ can be identified with the critical locus of $f$ inside $\mf M$.

One can rephrase this as follows: the cotangent stack $T^*\mf M$ (i.e. the total space of the cotangent bundle of $\mf M$, which is naturally a symplectic stack) contains two Lagrangian substacks isomorphic to $\mf M$, one of them obtained by embedding $\mf M$ into it via the zero section, and the other by embedding via the section $d_{dR}f$. Then $\tmc 3$ is the intersection of these two Lagrangians inside $\mf M$, so we have a pullback diagram
\begin{equation}\label{eq:crit3}
\begin{tikzcd}
	{\tmc 3} & {\mf M} \\
	{\mf M} & {T^*\mf M}
	\arrow["d_{dR}f", from=1-2, to=2-2]
	\arrow["z", from=2-1, to=2-2]
	\arrow[from=1-1, to=2-1]
	\arrow[from=1-1, to=1-2]
	\arrow["\lrcorner"{anchor=center, pos=0.125}, draw=none, from=1-1, to=2-2].
\end{tikzcd}
\end{equation}
Being a Lagrangian intersection, the \textit{derived} stack $\mdc 3 := \mf M \times^h_{T^*\mf M} \mf M$ carries a canonical $(-1)$-shifted symplectic structure by the following theorem:
\begin{thm}[{\cite[Theorem 2.9]{ptvv}}]\label{thm:lagint}
    Let $\bf F$ be an $n$-shifted symplectic derived Artin stack and ${\bf X} \to {\bf F}$, ${\bf Y} \to {\bf F}$ two morphisms of derived Artin stacks carrying Lagrangian structures. Then the derived stack ${\bf X} \times^h_{\bf F} {\bf Y}$ has a canonical $(n-1)$-shifted symplectic structure.
\end{thm}

On the other hand, we also have the following theorem.
\begin{thm}[{\cite[Theorem 0.1]{ptvv}, \cite[\S 1]{bd19}}]\label{thm:bdmain}
    Let $X$ be a smooth Calabi-Yau variety of dimension $n$. Then the derived stack $\ensuremath{\boldsymbol{\mathcal{M}}^d(X)}$ has a canonical $(2-n)$-shifted symplectic structure.
\end{thm}
For $\mdc 3$, the two $(-1)$ shifted symplectic structures guaranteed by the two theorems above are equivalent (see \cite[\S 4]{katzshi}, \cite[\S 3.5]{ricolfisavvas}). 
Thus \Cref{eq:crit3} can be interpreted as the special case of (the classical version of) a local converse to \Cref{thm:lagint}: that any $(-1)$-shifted symplectic derived scheme is locally a Lagrangian intersection.  
Motivated by these results about $\boldsymbol{\mathcal{M}}^d(\mathbb{C}^3)$ and $\boldsymbol{\mathcal{M}}^d(X)$, we are interested in the following questions in this paper:
\begin{que}\label{q1}
    Does a picture similar to \Cref{eq:crit3} exist for $\mdc 4$, i.e. can $\mdc 4$ be realized as a derived critical locus?
\end{que}
 Clearly, it cannot be \textit{too} similar: $\mdc 4$ has a $(-2)$-shifted symplectic structure by \Cref{thm:bdmain}, not a $(-1)$-shifted one. 
More generally:
 \begin{que}\label{q2}
     Given a $(-2)$-shifted symplectic derived scheme, when can we say that it is the derived intersection of two Lagrangians inside a $(-1)$-shifted symplectic derived scheme?
 \end{que}

We obtain the following answers:
\begin{itemize}
    \item for \Cref{q1}: ``Yes, if we replace $\mf M$ with a derived stack, and $T^*\mf M$ with its \textit{shifted cotangent stack}'';
    \item for \Cref {q2}: ``Locally, always!''
\end{itemize}

The answer to \Cref{q1} is a consequence of our first main result. 
\begin{thm}(\Cref{prop:evencase})\label{thm:evenmain}
   Let $\bf X$ be an affine $(-2)$-shifted symplectic scheme in Darboux form which has an even number of degree $(-1)$ generators. Then $\bf X$ is equivalent to the derived critical locus of a shifted function defined on a quasi-smooth derived scheme. 
\end{thm}

While working on this paper, the second author was informed at the CATS7 conference at CIRM that the above is a folklore result in the field. We are very grateful for all the conversations and comments on this.

The answer to \Cref{q2} is our second main result.

\begin{thm}(\'etale version: \Cref{thm:main}, Zariski version: \Cref{thm:mainzariski})\\\label{thm:generalmain}
Let $\bf X$ be a $(-2)$-shifted symplectic scheme. Then, locally $\bf X$ is equivalent to the derived intersection of two Lagrangians inside a $(-1)$-shifted symplectic derived scheme which is the shifted cotangent stack of a smooth classical scheme. 
\end{thm}
This result gives a local converse to \Cref{thm:lagint} for general $(-2)$-shifted symplectic derived schemes. It is worth noting that the proofs and constructions in both theorems are different even for the class of derived schemes where both apply.

As an application of the first main result, we also obtain a characterization of $\mdc 4$ as a derived critical locus:
\begin{thm}(\Cref{thm:maindim4})
    The derived moduli stack of length $d$ sheaves $\mdc 4$ admits a smooth atlas $\BSpec A_d \to \mdc 4$ such that $\BSpec A_d$ embeds into $\BSpec B_d$ for a closed subalgebra $B_d \subset A_d$ ,and $\BSpec B_d$ is \'etale locally equivalent to the derived critical locus of a shifted function.
\end{thm}

\subsection{Generalizations}\label{sec:generalizations}
Here are some potential generalizations of the results in this paper, which we leave for future exploration. 
\begin{itemize}

\item We expect all results about $(-2)$-shifted symplectic derived schemes can be extended to $(-2)$-shifted symplectic derived stacks in a manner analogous to the generalization of the Darboux Theorem in \cite{bbj} to the one in \cite{bbbbj}. Consequently, \Cref{thm:maindim4} can be rephrased as a statement about $\mdc{4}$ \textit{itself} being a derived critical locus (instead of needing to work on an atlas).

\item Preliminary calculations show that similar adaptation techniques can be applied to the Darboux Theorems in \cite{bbbbj} to obtain a general result along the lines of: \textit{``given $n < 0$, any $n$-shifted symplectic derived stack is locally equivalent to a Lagrangian intersection in an $(n+1)$-shifted symplectic derived stack''}. We have not included such a general result here but expect it to be true.
\end{itemize}
\subsection{Notation}
We work in the derived algebraic geometry framework introduced in \cite{hag1, hag2} and elsewhere. Here are some of the conventions we adopt in this paper:
\begin{itemize}
    \item We work over a field $k$ which is algebraically closed with characteristic $0$.
    \item All schemes are of finite type over $k$. We often refer to such schemes as \textit{classical} or ordinary to \textit{distinguish} them from derived schemes.
    \item We use the abbreviation \textit{dga} to refer to a differential graded algebra and \textit{cdga} to refer to a graded-commutative differential algebra. All such algebras $A$ are connective (i.e. $H^i(A) = 0$ for $i > 0$), and the differentials $d: A^m \to A^{m+1}$ are of degree +1.
    \item We use boldface letters to denote derived schemes and stacks e.g. {\bf X}, {$\bf \mc M$}. Accordingly, we use $\BSpec$ to denote affine derived schemes, and $\Spec$ to denote classical affine schemes.
    \item We denote by $t_0$ the functor that assigns to any derived stack its classical truncation.
    \item We denote by $QCoh({\bf F})$ the differential graded category (\textit{dg-category}) of quasicoherent sheaves on the derived stack $\bf F$, and by $L_{qcoh}({\bf F})$ its homotopy category; the latter is a triangulated category and, when $\bf F$ is equivalent to a classical scheme, is equivalent to its ordinary derived category of complexes of quasicoherent sheaves.
\end{itemize}

\subsection{Acknowledgments}
We are very grateful to Christopher Brav, Sheldon Katz, Sarunas Kaubrys, Young-Hoon Kiem, Tony Pantev, Hyeonjun Park and Marco Robalo for several illuminating conversations, suggestions and ideas. The first author was supported in part by NSF grants DMS-1802242 and DMS-2201203. The second author would like to acknowledge the organizers of the CATS7 conference at CIRM, where she was able to speak on relevant work and learned very helpful comments from the participants.

\section{Realization as a Lagrangian intersection}
In this section we prove our main results \Cref{thm:evenmain} and \Cref{thm:generalmain}. We begin by recalling the following definitions (where we have used some notation from \cite{ptvv}):
\begin{defn}(Definition 2.7, \cite{ptvv})
Let $f:{\bf X}\rightarrow {\bf F}$ be a morphism of derived Artin stacks and $\omega$ an $n$-shifted symplectic form on $\bf F$. An \textbf{isotropic structure on $f$ (relative to $\omega$)} is a path between $0$ and $f^*\omega$ in the space $\A^{2,cl}({\bf X},n)$. The space of isotropic structures on $f$ (relative to $\omega$) is defined to be the path space
\[
\text{Isot}(f,\omega):=Path_{0,f^*\omega}(\A^{2,cl}({\bf F},n)).
\]
\end{defn}
\begin{defn}(Definition 2.8, \cite{ptvv})
Let $f:{\bf X}\rightarrow {\bf F}$ be a morphism of derived Artin stacks and $\omega$ an $n$-shifted symplectic form on $\bf F$. An isotropic structure $h$ on $f$ is a \textbf{Lagrangian structure on $f$ (relative to $\omega$)} if the induced morphism
\[
\Theta_h: \bT_f \rightarrow \bL_{\bf X}[n-1]
\]
is a quasi-isomorphism of perfect complexes.
\end{defn}
\begin{defn}
    A pair $({\bf Z}, \omega)$ i.e. a derived scheme ${\bf Z}$ with an $(n-1)$-shifted symplectic form $\omega$ is called a \textbf{Lagrangian intersection} if there exists an $n$-shifted symplectic derived scheme $\bf F$ with Lagrangian structures on ${\bf X} \to {\bf F}$ and ${\bf Y} \to {\bf F}$, and morphisms ${\bf Z} \to {\bf X}$ and ${\bf Z}\to {\bf Y}$, inducing an equivalence ${\bf Z} \simeq {\bf X}\times^h_{\bf F} {\bf Y}$ such that the canonical $(n-1)$-shifted symplectic structure on $\bf Z$ is equivalent to $\omega$. We also say $\bf Z$ \textbf{is a Lagrangian intersection inside} $\bf X$.
\end{defn}
Here the canonical $(n-1)$-shifted symplectic structure on $\bf Z$ is the one in Theorem \ref{thm:lagint}. Here is the main result we will be proving in this section: 
\begin{thm} \label{thm:main}
    Let $\bf X$ be a derived scheme with a $(-2)$-shifted symplectic form $\omega$, and $x \in \bf X$. Then there exists a standard form cdga $A$ minimal at $p \in \Spec H^0(A)$ and an \'etale morphism $f:\BSpec A \to \bf X$ such that $f(p) = x$, and $(\BSpec A, f^*\omega)$ is a Lagrangian intersection.
\end{thm}
The proof will proceed based on the Darboux form construction of \cite{bbj} as follows: the Darboux theorem \cite[Theorem 5.18]{bbj} guarantees the existence of $f, p$ and $A$ satisfying all the requirements above except the part about $(\BSpec A, f^*\omega)$ being a Lagrangian intersection. That is the part we focus on. We treat a special case, the one with an even number of generators in degree $(-1)$, first because of the straightforward nature of the construction. The general case is handled thereafter.

\subsection{The even case}
We first prove the theorem for the case where the number of generators of $A$ in degree $(-1)$ is even. 
Let $A$ be a cdga in Darboux form with a $(-2)$-shifted symplectic structure as in \cite[Remark 5.11]{bbj} i.e. let $m_0, m_1 \in \mb N$ and $A(0)$ be a smooth algebra of dimension $m_0$ over $k$, with $d_{dR}x_1, \ldots d_{dR}x_{m_0}$ forming a basis for $\Omega^1_{A(0)}$ over $A(0)$ for some $x_1,\ldots, x_{m_0} \in A(0)$; let $A$ be the commutative graded algebra freely generated over $A(0)$ by the variables
\begin{align*}
    y_1, \ldots, y_{m_1}, z_1, \ldots, z_{m_1} &\qquad\text{in degree -1}\\
    w_1, \ldots, w_{m_0} &\qquad\text{in degree -2}
\end{align*}
Define the Hamiltonian $\Phi \in A^{-1}$ by
\begin{equation}
    \Phi = \sum_{i=1}^{m_1} f_i y_i + g_i z_i
\end{equation}
where $f_1, \ldots, f_{m_1}, g_1, \ldots, g_{m_1}$ are elements of $A(0)$. The differential on $A$ is now defined as follows: $d$ is zero on $A(0) = A^0$;
\begin{align*}
    dy_i = \frac{\partial \Phi}{\partial z_i} = g_i && dz_i = \frac{\partial \Phi}{\partial y_i} = f_i\\
    &dw_j = \frac{\partial \Phi}{\partial x_j} = \sum_{i=1}^{m_1} \frac{\partial f_i}{\partial x_j} y_i + \frac{\partial g_i}{\partial x_j} z_i&
\end{align*}
The equality $d(dw_j) = 0$ is obtained by applying the operator $\partial/\partial x_j$ to the classical master equation:
\begin{equation}
    \sum_{i=1}^{m_1} f_i g_i = 0.
\end{equation}
Now define
\begin{equation}
    \omega^0 = \sum_{i=1}^{m_1}d_{dR}y_i\; d_{dR}z_i + \sum_{j=1}^{m_0}d_{dR}x_j\; d_{dR}w_j.
\end{equation}
Then $\omega := (\omega^0, 0, \ldots, 0, \ldots)$ is a $(-2)$-shifted symplectic structure on ${\bf X} := \BSpec A$.

Now we define a cdga $B$. The algebra $B(0)$ equals to $A(0)$. Let $B$ be the graded algebra freely generated over $B^0 = B(0) = A(0)$ by $y_1,\ldots, y_{m_1}$ in degree $-1$. Define the differentials to be $dy_i = g_i$. We adopt the notation ${\bf Y} := \BSpec B$. Then $\bf Y$ is quasi-smooth.

Let $T^*[-1] {\bf Y} := \BSpec \Sym_{\bf Y} (\mb T_{\bf Y}[1])$ be the $(-1)$-shifted cotangent stack of $\bf Y$. Then we have that $T^*[-1]{\bf Y} = \BSpec C$, where $C$ is the cdga freely generated over $B$ by the variables
\begin{align*}
    \beta_1, \ldots, \beta_{m_1} &\qquad\text{in degree 0}\\
    \alpha_1, \ldots, \alpha_{m_0} &\qquad\text{in degree -1},
\end{align*}
where we define $\alpha_j := \left(d_{dR}x_j\right)^\vee[1]$ and $\beta_i := \left(d_{dR}y_i\right)^\vee[1]$, with
\[  
d\alpha_j = -\sum_{i=1}^{m_1} \frac{\partial g_i}{\partial x_j} \beta_i.
\]
We define the following $(-1)$-shifted 2-form on $T^*[-1]{\bf Y}:$
\[ \omega_L^0 := \sum_{i=1}^{m_1}d_{dR}y_i\; d_{dR}\beta_i - \sum_{j=1}^{m_0}d_{dR}x_j\; d_{dR}\alpha_j, \]
which is the de Rham differential of the canonical ``Liouville'' form.
\begin{prop}\cite[Proposition 1.21]{ptvv}
    The $(-1)$-shifted closed 2-form \[\omega_L := \left(\omega^0_L, 0, \ldots, 0, \ldots\right)\] is shifted symplectic.
\end{prop}
Thus $T^*[-1]{\bf Y}$ carries a canonical $(-1)$-shifted symplectic structure. Now consider
\[ \Psi := \sum_{i=1}^{m_1} f_iy_i \in B^{-1}. \]
Then $d_{dR}\Psi$ is a section of the $(-1)$-shifted cotangent bundle of $\bf Y$ given by the map
\begin{align*}
    d_{dR}\Psi: C &\to B,\\
    B(0) &\mapsto B(0),\\
    y_i &\mapsto y_i,\\
    \alpha_j &\mapsto \sum_{i=1}^{m_1} \frac{\partial f_i}{\partial x_j}y_i,\\
    \beta_i &\mapsto f_i.
\end{align*}
In particular, the morphism $d_{dR}\Psi: {\bf Y} \to T^*[-1]{\bf Y}$ carries a Lagrangian structure. The zero section $z: {\bf Y} \to T^*[-1]{\bf Y}$ naturally carries a Lagrangian structure. Thus by the theorem, the derived intersection
\[ \mb R Crit(\Psi) := {\bf Y} \times^h_{T^*[-1]{\bf Y}} {\bf Y}, \]
where the two morphisms are $d_{dR}\Psi$ and $z$, carries a canonical $(-2)$-shifted symplectic structure, $Res(\omega_L, d_{dR}\Psi, z)$, called the \textit{residue of} $\omega_L$. We will also denote it by $R(\omega_L)$ when there is no room for confusion.

\begin{prop}\label{prop:evencase}
    There is a map $\kappa: {\bf X} \to \mb R Crit(\Psi)$ which is an equivalence of derived schemes. Moreover, we have (up to a constant) $\kappa^*R(\omega_L) \sim \omega$.
\end{prop}
\begin{proof}
    We first compute $\mb R Crit(\Psi) = \BSpec \left(B \otimes^\mb L_C B\right)$. For this we need a cofibrant replacement of (at least) one of the $B$'s as a $C$-cdga. Consider the semifree $C$-cdga $D := C[\theta_i, \tau_j]$, where $\theta_1, \ldots, \theta_{m_1}$ are in degree $(-1)$ and $\tau_1, \ldots, \tau_{m_0}$ are in degree $(-2)$. Set the differentials $d\theta_i = \beta_i$ and $d\tau_j = \alpha_j + \sum_{i=1}^{m_1} \frac{\partial g_i}{\partial x_j} \theta_i$. Recall that we denote the zero section ${\bf Y}\rightarrow T^*[-1]{\bf Y}$ by $ z$. Then the map $z: C \to B$ induces a map $D\to B$ when we map the $\theta_i$'s and $\tau_j$'s to zero. This map is a quasi-isomorphism. Thus $D$ is a cofibrant replacement for $B$ as a $C$-cdga (with the $C$-action on $B$ given by $z$) and we have $\mb R Crit (\Psi) \simeq \BSpec \left( B \otimes_C D \right)$. Since $D$ is freely generated over $C$, it follows that $B\otimes_C D$ is freely generated over $B$, with generators $\theta_1, \ldots, \theta_{m_1}$ in degree $(-1)$ and $\tau_1,\ldots, \tau_{m_0}$ in degree $(-2)$, and differentials given by $d\theta_i = f_i$ and $d\tau_j = \sum_{i=1}^{m_1} \frac{\partial f_i}{\partial x_j}y_i + \sum_{i=1}^{m_1} \frac{\partial g_i}{\partial x_j} \theta_i$.

    Now consider the map
    \begin{align*}
        \kappa: B\otimes_C D &\to A\\
        B(0) &\mapsto B(0)\\
        y_i &\mapsto y_i\\
        \theta_i &\mapsto z_i\\
        \tau_j &\mapsto w_j
    \end{align*}
    It can be checked that $\kappa$ is a quasi-isomorphism (in fact, an isomorphism) of cdgas. Denoting the induced map ${\bf X} = \BSpec A \to \mb R Crit(\Psi)$ also by $\kappa$, we obtain the required map.

    It remains to show the assertion about the shifted symplectic structures. Denoting the map $C \to D$ induced by $z: C \to B$ also by $z$, we see that
    \[ z_*\omega_L^0 = \sum_{i=1}^{m_1}d_{dR}y_i\; d_{dR}\beta_i - \sum_{j=1}^{m_0}d_{dR}x_j\; d_{dR}\alpha_j. \]
    Now consider the form $\nu = \left(\nu^0, 0, \ldots, 0\right)$ on $D$, where
    \[ \nu^0 = -\sum_j d_{dR}x_j d_{dR}\tau_j - \sum_i d_{dR} y_i d_{dR} \theta_i. \]
    Then $d\nu^0 =z_*\omega^0_L$ and $d_{dR}\nu^0 = 0$, so that $\nu$ is an isotropic structure for $\omega_L$ on the map $z$.

    To understand the Lagrangian structure on the other morphism better, we take another semi-free resolution $D'$ of $B$, where $D' := C[\theta'_i, \tau'_j]$, where $\theta'_1, \ldots, \theta'_{m_1}$ are in degree $(-1)$ and $\tau'_1, \ldots, \tau'_{m_0}$ are in degree $(-2)$. Set the differentials $d\theta'_i = f_i - \beta_i$ and $d\tau'_j = \sum_{i=1}^{m_1} \frac{\partial f_i}{\partial x_j}y_i + \sum_{i=1}^{m_1} \frac{\partial g_i}{\partial x_j} \theta_i' - \alpha_j$. The induced map $D' \to B$ is a quasi-isomorphism. Now
    \[ \left(d_{dR}\Psi\right)_*\omega_L^0 = \sum_{i=1}^{m_1}d_{dR}y_i\; d_{dR}\beta_i - \sum_{j=1}^{m_0}d_{dR}x_j\; d_{dR}\alpha_j \]
    on $D'$.
    Consider the form $\mu = \left(\mu^0, 0, \ldots, 0\right)$ on $D'$, where
    \[ \mu^0 = \sum_j d_{dR}x_j d_{dR}\tau'_j + \sum_i d_{dR} y_i d_{dR} \theta'_i. \]
    Then $d\mu^0 = \left(d_{dR}\Psi\right)_*\omega^0_L$ and $d_{dR}\mu^0 = 0$, so that $\mu$ is an isotropic structure for $\omega_L$ on the map $d_{dR}\Psi$. Denoting the map $D'\otimes_C B \to A$ by $\kappa'$, we have that $R(\omega_L) = \kappa'_* \mu - \kappa_* \nu = 2\omega$, so the result will follow once we show $\nu$ and $\mu$ are Lagrangian structures and are equivalent to the canonical Lagrangian structures on their respective morphisms.

    For the map $z: C \to D$, we have $\mb T_z \simeq B\langle \frac{\partial}{\partial \theta_i}, \frac{\partial}{\partial \tau_j}\rangle$ as graded $B$-modules, while $\mb L_{\bf Y}[-2] \simeq \mb L_B[-2] \simeq B\langle d_{dR} x_j, d_{dR} y_i\rangle$ as graded $B$-modules. Then we see that the map induced by $\nu$ is
        \begin{align*}
        \Theta_\nu : \mb T_z &\to \mb L_B [-2]\\
        \frac{\partial}{\partial \tau_j} &\mapsto d_{dR} x_j\\
        \frac{\partial}{\partial \theta_i} &\mapsto d_{dR} y_i
    \end{align*}
    which is a quasi-isomorphism. 
    
    To see that this is equivalent to the canonical structure on the morphism $z$, we construct another cdga $\tilde{B}$ which is quasi-isomorphic to $B$. Let $\tilde{B}(0) := C(0) = B(0)[\beta_j]$, and let the generators of $\tilde{B}$ in degree $-1$ be $y_1, \ldots, y_{m_1}, \theta_1, \ldots, \theta_{m_1}$, with $dy_i = g_i$ and $d\theta_i = \beta_i$. Then $\tilde{B}$ is quasi-isomorphic to $B$ and the induced map $z: C \to \tilde{B}$ is the identity on $x_j, \beta_i, y_i$, while we have $\alpha_j \mapsto \sum_i \frac{\partial g_i}{\partial x_j} \theta_i$. We are now in the situation of \cite[Example 3.6]{joycesafronov}, with the superpotential $\varepsilon \in \tilde{B}^{-1}$ given by
    \[ \varepsilon = - \sum_i \left(g_i \theta_i + y_i \beta_i\right). \]
    Note that we have rewritten the ``quadratic form'' in degree $-1$ here as a nondegenerate pairing since we have an even number of generators in degree $-1$. The Lagrangian structure on $z$ is equivalent to $-\sum_i d_{dR} y_i d_{dR} \theta_i$, which is equivalent to the restriction of $\nu^0$ to $\tilde{B}$. We conclude that $\nu$ is equivalent to the canonical Lagrangian structure on the section $z$.
    
    A similar calculation shows that $\mu$ is a Lagrangian structure, and is equivalent to the canonical Lagrangian structure on the section $d_{dR}\Psi$, and we're done.
\end{proof}
This completes the proof in the even case. Observe that we have proved something stronger than \Cref{thm:generalmain} for this case: the local description is not just as a Lagrangian intersection, but as a \textit{derived critical locus}, which is a special type of Lagrangian intersection.

\subsection{The general case}
We now proceed to prove the general case. As will be clear in the proof, this is not a generalization of the proof in the even case, and is in fact a completely independent proof even for the situation where we have an even number of generators in degree $(-1)$.

Let $A$ be a cdga in the strong Darboux form in \cite{bbj} with a $(-2)$-shifted symplectic structure. In particular, $A$ is of the form in Example 5.10 in \cite{bbj}. We first recall the setup. Let $m_0$, $A(0)$, be as in Section 2.1. We denote the degree $-1$ generators by $y_1, ...y_{m_1}$, and the degree $-2$ generators by $z_1,..., z_{m_0}$.
Recall the classical master equation is given by $\frac{1}{4}\sum_{j=1}^{m_1}(\frac{\partial\Phi}{\partial y_j})^2=0$.  Choose a Hamiltonian $\Phi\in A^{-1}$ that satisfies the classical master equation, we write
\[
\Phi=\sum_{j=1}^{m_1}f_jy_j,
\]
where $f_i\in A(0)$ and $\sum_i f_i^2 = 0$.
Let $h_j:=\frac{1}{2}f_j$, and $g_k^j:=\frac{\partial f_j}{\partial x_k}$. The differentials on $A$ are given by 
\begin{equation*}
\begin{split}
dy_j&=h_j,\\
dz_k&=\sum_jg_k^jy_j.
\end{split}
\end{equation*}
For the cdga $A$, let $A(1)$ be the cdga freely generated over $A(0)$ by the degree $-1$ generators, with the same differential. 
Consider the following fiber product diagram 
\begin{equation}\label{eq:lagint}
	\begin{tikzcd}
 \BSpec A(0)\times^h_{T^*[-1]\BSpec A(0)}\BSpec A(1) &\BSpec A(0)\\
 \BSpec A(1) & T^*[-1]\BSpec A(0)
        \arrow[from=1-1, to=1-2]
	\arrow[from=1-1, to=2-1]
	\arrow["p", from=1-2, to=2-2]
	\arrow["q", from=2-1, to=2-2]
		\end{tikzcd}
\end{equation}
We denote $T^*[-1]\BSpec A(0)$ by $\BSpec C_0$. Then 
$C_0 \simeq A(0)[\alpha_k]$,  
where $\alpha_k=(d_{dR}x_k)^\vee[1]$ are generators of degree $-1$. 
The map $p$ is the zero section map, and the map $q$ is induced by the map of cdgas $C_0\rightarrow A(1)$:
\[
\begin{split}
A(0)&\mapsto A(0),\\
\alpha_k&\mapsto \sum_jg_k^j y_j
\end{split}
\]
Note that this is indeed a map of cdgas by applying $\frac{\partial}{\partial x_k}$ to the classical master equation.
\begin{prop}
There exists a map $\kappa:\BSpec A\rightarrow \BSpec A(0)\times^h_{T^*[-1]\BSpec A(0)}\BSpec A(1)$ which is an isomorphism. 
\end{prop}
\begin{proof}
For the first claim, it is enough to show that 
\begin{equation}
\label{Equ:Lagintcdga}
\begin{tikzcd}
	{C_0} & {A(1)} \\
	{A(0)} & A
	\arrow["q'", from=1-1, to=1-2]
	\arrow["p'", from=1-1, to=2-1]
	\arrow[from=1-2, to=2-2]
	\arrow[from=2-1, to=2-2]
\end{tikzcd}
\end{equation}
is a homotopy pushout diagram. For this we construct a cofibrant replacement for $A(0)$ as a $C_0$-algebra. Recall that $C_0 = A(0)[\alpha_k]$. Consider the semi-free dg $C_0$-module $M := C_0 \oplus C_0\langle \tau_k \rangle$, 
with $\tau_k$ in degree $-2$ and $d\tau_k = \alpha_k$. Then, as a chain complex, we have
\begin{equation}
\label{Equ:dgresA0}
M =...\rightarrow A(0)\la\tau_k, \alpha_{k_1}\alpha_{k_2}\ra\rightarrow A(0)\la \alpha_k \ra\rightarrow A(0)
\end{equation}
where $\alpha_k\mapsto 0$ and $\tau_k\mapsto \alpha_k$.
By construction $M$ is quasi-isomorphic to $A(0)$. Then 
\[
A(0)\otimes^\bL_{C_0} A(1)\simeq M\otimes_{C_0} A(1)
\]
which is quasi-isomorphic to $A(1)\oplus A(1)\langle \tau_k\rangle$. As a chain complex $M\otimes_{C_0}A(1)$ is given by
\[
...\rightarrow A(0)\la\tau_k, y_{j_1}y_{j_2}\ra\rightarrow A(0)\la y_j\ra\rightarrow A(0)
\]
where $y_j\mapsto h_j$,
and $\tau_k\mapsto \sum_j g_k^jy_j$.
Then the map 
$A\rightarrow M\otimes_{C_0}A(1)$ where $z_k\mapsto \tau_k$ is an isomorphism of cdgas.
\end{proof}

Consider the $(-1)$-shifted symplectic form on $T^*[-1]\BSpec A(0)$: $\omega_L=(\omega_L^0,...)$, where 
\[
\omega_L^0=\sum_id_{dR}x_id_{dR}\alpha_i.
\]
We explore the Lagrangian structure on $p$ and $q$ (relative to $\omega_L$). We begin with the Lagrangian structure on $q$. Consider the following $2$-form of degree $-2$:
\[
\nu=-\sum_j\dd y_j\dd y_j\in \left(\bigwedge^{2}\Omega_{A(1)}^1\right)^{-2}.
\]
Then we have 
\[
\label{Equ:dnu}
d\nu
=\sum_j\sum_k 2\frac{\partial h_j}{\partial x_k}\dd x_k\dd y_j.
\]
On the other hand, we have
\[
\begin{split}
q^*\omega_L&=q^*\sum_k\dd x_k\dd \alpha_k\\
&=\sum_k\sum_jg_k^j\dd x_k\dd y_j
\end{split}
\]
Hence $d\nu=q^*\omega_L$. Then by \cite[(2.4)]{joycesafronov}, 
we have $\nu\in \text{Isot}(q, \omega_L)$. To show $\nu$ is a Lagrangian structure on $q$ relative to $\omega_L$, we need to show $\Theta_\nu:\bT_q\rightarrow \bL_{\BSpec A(1)}[-2]$ is a quasi-isomorphism. 

The tangent complex $\mb T_{\BSpec A(1)}$ is, as an $A(1)$-module, given by $\mb T_{\BSpec A(1)} = A(1)\langle \frac{\partial}{\partial x_i}, \frac{\partial}{\partial y_j} \rangle$, with $\frac{\partial}{\partial x_i}$ in degree $0$ and $\frac{\partial}{\partial y_j}$ in degree $1$. While $\mb T_{\BSpec C_0} = C_0\langle \frac{\partial}{\partial x_i}, \frac{\partial}{\partial \alpha_j} \rangle$, with $\frac{\partial}{\partial x_i}$ in degree $0$ and $\frac{\partial}{\partial \alpha_j}$ in degree $1$.
The map $q: \BSpec A(1) \to \BSpec C_0$ induces the exact triangle in $L_{qcoh}(\BSpec A(1))$ :\[
\bT_q\xrightarrow{a} \bT_{\BSpec A(1)}\xrightarrow{b} q^*\bT_{\BSpec C_0}\xrightarrow{+1}
\]
where $b$ is defined by
\[
\begin{split}
\frac{\partial}{\partial x_i}&\mapsto q^*\frac{\partial}{\partial x_i}+\sum_k\sum_j\frac{\partial g_k^j}{\partial x_i}y_j q^*\frac{\partial}{\partial \alpha_k}, \\
\frac{\partial}{\partial y_j}&\mapsto \sum_k g_k^jq^*\frac{\partial}{\partial\alpha_k}.
\end{split}
\]
Then $\bT_q \simeq \text{Cone}(b)[-1]$. We write the degree $0$, $1$ and $2$ terms of $\bT_q$ in the following explicit form:
\begin{equation}
\label{Equ:Tq}
\begin{array}{c}
A(0)\la\frac{\partial}{\partial x_i}\ra\\
\oplus\\
A(0)\la y_j\frac{\partial}{\partial y_k}\ra\\
\oplus\\
A(0)\la y_m y_nq^*\frac{\partial}{\partial\alpha_l}\ra\\
\oplus\\
A(0)\la y_s q^*\frac{\partial}{\partial x_i}\ra
\end{array}
\xrightarrow{
\begin{bmatrix}
\frac{\partial h_k}{\partial x_i} & h_j &0 &0\\
-\frac{\partial g_l^j}{\partial x_i} & g_l^k& h_m\delta_{jn} - h_n\delta_{jm} & 0\\
-id & 0 & 0 & h_s
\end{bmatrix}
}
\begin{array}{c}
A(0)\la\frac{\partial}{\partial y_k}\ra\\
\oplus\\
A(0)\la y_jq^*\frac{\partial}{\partial \alpha_l}\ra\\
\oplus\\
A(0)\la q^*\frac{\partial}{\partial x_i}\ra
\end{array}
\xrightarrow{\begin{bmatrix}
-g_l^k & h_j & 0
\end{bmatrix}}
\begin{array}{c}
A(0)\la q^*\frac{\partial}{\partial\alpha_l}\ra
\end{array}
\end{equation}

The general terms of degree $-\lambda$ and degree $-\lambda+1$ and the maps between them are defined by linearity over $A(1)$.

Note that $\bL_{\BSpec A(1)}$ is a free $A(1)$-module generated by $\dd y_i$'s in degree $-1$ and $\dd x_i$'s in degree $0$. Then the homotopy between $q^*\omega_L:\bT_q\rightarrow \bL_{\BSpec A(1)}[-1]$ and $0$ induced by $\nu$ is given by 

\[
\frac{\partial}{\partial y_j}\mapsto \iota_{\frac{\partial}{\partial y_j}}\nu=-2\dd y_j[-1]\\
\]
where $\iota_{\frac{\partial}{\partial y_j}}$ is contraction with $\frac{\partial}{\partial y_j}$, and the rest of the maps are all zero. 

On the other hand, there is a canonical homotopy from $\bT_q\rightarrow q^*\bT_{\BSpec C_0}$ to $0$. 
We denote the homotopy from $q^*\omega_L$ to $0$ induced by this canonical homotopy by $c$. Then $c$ is given by 

\begin{equation}
\label{Equ:c}
\begin{split}
q^*\frac{\partial}{\partial \alpha_k}[-1]&\mapsto -\dd x_k[-1],\\
q^*\frac{\partial}{\partial x_i}[-1]&\mapsto \sum_jg_i^j\dd y_j[-1]-\sum_j\sum_k \frac{\partial g_i^j}{\partial x_k}y_j\dd x_k[-1],
\end{split}
\end{equation}
and all the other maps are zero.

Combining these two homotopies, we defined $\Theta_\nu=c-\nu$.

\begin{prop}
The isotropic structure $\nu$ is a Langrangian structure on $q$ relative to $\omega_L$.
 
\end{prop}

\begin{proof}
We only need to show that the map $\Theta_\nu:\bT_q\rightarrow \bL_{\BSpec A(1)}[-2]$ is a quasi-isomorphism.

We define a chain complex $\bT_q'$. The terms in each degree is the same as in the complex $\bT_q$, and the differentials between the first three terms are given by

\begin{equation}
\label{Equ:Tq0}
\begin{array}{c}
A(0)\la\frac{\partial}{\partial x_i}\ra\\
\oplus\\
A(0)\la y_j\frac{\partial}{\partial y_k}\ra\\
\oplus\\
A(0)\la y_my_nq^*\frac{\partial}{\partial\alpha_l}\ra\\
\oplus\\
A(0)\la y_s q^*\frac{\partial}{\partial x_i}\ra
\end{array}
\xrightarrow{
\begin{bmatrix}
0 & h_j &0 &0\\
0 & g_l^k& h_m\delta_{jn}-h_n\delta_{jm} & 0\\
id & 0 & 0 & h_s
\end{bmatrix}
}
\begin{array}{c}
A(0)\la\frac{\partial}{\partial y_k}\ra\\
\oplus\\
A(0)\la y_jq^*\frac{\partial}{\partial \alpha_l}\ra\\
\oplus\\
A(0)\la q^*\frac{\partial}{\partial x_i}\ra
\end{array}
\xrightarrow{\begin{bmatrix}
-g_l^k & h_j & 0
\end{bmatrix}}
\begin{array}{c}
A(0)\la q^*\frac{\partial}{\partial\alpha_l}\ra
\end{array}
\end{equation}

We form the following diagram
\begin{equation}
	\begin{tikzcd}
 	\bT_q^{-\lambda}\arrow[r, "d_q^{-\lambda}"] \arrow[d, "\phi^{-\lambda}"]&\bT_q^{-\lambda+1}\arrow[d, "\phi^{-\lambda+1}"]\\
    \bT_q'^{-\lambda}\arrow[r, "d_q'^{-\lambda}"] &\bT_q'^{-\lambda+1},
		\end{tikzcd}
\end{equation}
where $\phi:=(..., \phi^0, \phi^1, \phi^2)$ is an $A(1)$ linear morphism defined by 
\[
\begin{split}
\phi^2 &=id,\\
\phi^1 &=\begin{bmatrix}
id & 0 & \frac{\partial h_k}{\partial x_i}\\
0 & id & -\frac{\partial g_l^j}{\partial x_i}\\
0 & 0 & -id
\end{bmatrix},\\
\phi^{-\lambda} &=\begin{bmatrix}
id & 0 & 0 & 0\\
0 & id & 0 &  \frac{\partial h_k}{\partial x_i}\\
0 & 0 & id & -\frac{\partial g_l^j}{\partial x_i}\\
0 & 0 & 0 & -id
\end{bmatrix}
\end{split}
\]

for $\lambda\geq 0$.
Note that this is indeed a commutative diagram. The only nontrivial term is the image of $y_{I_{\lambda+1}}q^*\frac{\partial}{\partial x_i}$.
Since $\phi^{-\lambda}$ is $A(1)$ linear for any $\lambda$, we only need to check 
\begin{equation}
\label{Equ:commdiag}
\phi^{\lambda+1}\circ d_q\left(q^*\frac{\partial}{\partial x_i}\right)=d_q'\circ\phi^\lambda\left(q^*\frac{\partial}{\partial x_i}\right).
\end{equation}
The left hand side is given by
\[
\label{comm1sim}
\begin{split}
\phi^{\lambda+1}\circ d_q\left(q^*\frac{\partial}{\partial x_i}\right)=\phi^{\lambda+1}(0)=0.
\end{split}
\]
The right hand side is given by
\[
\label{comm2sim}
\begin{split}
d'_q\circ\phi^\lambda\left(q^*\frac{\partial}{\partial x_i}\right)
&=d'_q\left(
\sum_k\frac{\partial h_k}{\partial x_i}\frac{\partial}{\partial y_k}
-\sum_{j}\sum_l\frac{\partial g_l^j}{\partial x_i}y_jq^*\frac{\partial}{\partial \alpha_l}
-q^*\frac{\partial}{\partial x_i}\right)
\\
&=-\sum_k\sum_l\frac{\partial h_k}{\partial x_i}g_l^kq^*\frac{\partial}{\partial \alpha_l}-\sum_{j}\sum_l\frac{\partial g_l^j}{\partial x_i}h_jq^*\frac{\partial}{\partial \alpha_l}\\
&=0.
\end{split}
\]

The last equality is obtained by applying $\frac{\partial}{\partial x_i}$ to the equation $\sum_k h_kg_l^k=0$, where the latter equation comes from $d\circ d=0$ in the cdga $A$.
Hence Equation \ref{Equ:commdiag} is satisfied.

Similarly, we also have a commutative diagram 
\begin{equation}
	\begin{tikzcd}
 	\bT_q'^{-\lambda}\arrow[r, "d_q'^{-\lambda}"] \arrow[d, "\phi^{-\lambda}"]&\bT_q'^{-\lambda+1}\arrow[d, "\phi^{-\lambda+1}"]\\
    \bT_q^{-\lambda}\arrow[r, "d_q^{-\lambda}"] &\bT_q^{-\lambda+1}
		\end{tikzcd}.
\end{equation}
It is easy to see that $\phi^i\circ \phi^i=Id$ for all $i\leq 2$,
hence $\phi$ induces a quasi-isomorphism between equation \ref{Equ:Tq} and \ref{Equ:Tq0}.

Recall that 
$\bL_{\BSpec A(1)}[-2]$ is generated by generators $d_{dR}x_i[-2]$'s in degree $2$ and generators $d_{dR}y_j[-2]$'s in degree $1$. We define an $A(1)$ linear morphism between \ref{Equ:Tq0} and $\bL_{\BSpec A(1)}[-2]$ by $\psi:=(...,\psi_0, \psi_1,\psi_2)$, where 
\[
\begin{split}
&\psi_1: \frac{\partial}{\partial y_j}\rightarrow \iota_{\frac{\partial}{\partial y_j}}\nu=2d_{dR}y_j[-2],\\ 
&\psi_2: q^*\frac{\partial}{\partial\alpha_i}[-1]\rightarrow -d_{dR}x_i[-2].
\end{split}
\]
Then $\psi$ is a quasi-isomorphism. Composing $\phi$ and $\psi$, we have $\Theta_\nu=\psi\circ\phi$ defines a quasi-isomorphism between \ref{Equ:Tq} and $\bL_{\BSpec A(1)}[-2]$. 
\end{proof}

It is relative straightforward to see that $p$ is equipped with a Lagrangian structure. We still write down this Lagrangian structure explicitly since it will also be useful when comparing the symplectic forms. 
Replace $A(0)$ by the cofibrant semi-free $C_0$ algebra $M$. Then diagram \ref{Equ:Lagintcdga} becomes:
\begin{equation}
\label{Equ:diagM}
\begin{tikzcd}
	{C_0} & {A(1)} \\
	{M} & M\otimes_{C_0}A(1)
	\arrow["q'", from=1-1, to=1-2]
	\arrow["p'", from=1-1, to=2-1]
	\arrow["a", from=1-2, to=2-2]
	\arrow["b", from=2-1, to=2-2].
\end{tikzcd}
\end{equation}
We abuse notation to denote the corresponding maps after applying $\BSpec$ also by $p'$, $q'$, and use the same notation for the variables in $C_0$ and $M$; this gives us $p^*\omega_L=\sum_id_{dR}x_id_{dR}\alpha_i$. 
Let \[
\delta:=\sum_id_{dR}x_id_{dR}\tau_i\in(\wedge^{2}\Omega_{A(0)})^{-2}.
\]
Since in $\text{DR}(M)$, we have
\begin{equation}
\label{Equ:delta}
d(d_{dR}x_id_{dR}\tau_i)=d_{dR}x_id_{dR}\alpha_i,
\end{equation}
by \cite[(2.4)]{joycesafronov}, we have $\delta\in \text{Isot}(p, \omega_L)$.

\begin{prop}
The isotropic structure $\delta$ is a Lagrangian structure on $p$ relative to $\omega_L$.
\end{prop}
\begin{proof}
Using explicit representatives, we have 
$\bT_p$ is isomorphic to the $M\simeq A(0)$-module generated by 
$p^*\frac{\partial}{\partial \alpha_l}[-1]$ in degree $2$. 
Then we see that the homotopy between $p^*\omega_L: \bT_p\rightarrow \bL_{\BSpec A(0)}[-1]$ and $0$ induced by $\delta$ is just $0$. 

On the other hand, the canonical homotopy between $p^*\omega_L$ and $0$ is given by 
\[
c': p^*{\frac{\partial}{\partial\alpha_l}}[-1]\rightarrow -d_{dR}x_l[-2],
\]
and the rest of the maps are all zero. 
Combining the two homotopies, we define 
\[
\Theta_\delta=c'-\delta.
\]
On the other hand, $\bL_{\BSpec A(0)}[-2]$ is generated by $d_{dR}x_i[-2]$ in degree $2$.
Then it is easy to see that 
\[\Theta_\delta:\bT_p\rightarrow\bL_{\BSpec A(0)}
\]
is a quasi-isomorphism.
\end{proof}

We denote the $(-2)$-shifted symplectic structure on $ \BSpec A(0)\times^h_{T^*[-1]\BSpec A(0)}\BSpec A(1)$ induced by Lagrangian intersections by $R(\omega_L)$. We denote the $(-2)$-shifted symplectic structure from the Darboux form on $\BSpec A$ by $\omega$. We denote $\omega_1\sim\omega_2$ if they are equivalent to each other in the sense of Definition 5.1 in \cite{bbj}. 
\begin{prop}
Using the notation as in the above paragraph, we have $\kappa^*R(\omega_L)\sim\omega$.
\end{prop}
\begin{proof}
Note that diagram \ref{Equ:diagM} commutes honestly (not just up to homotopy). Then the homotopy $u^*: (q\circ a)^*\rightarrow (p\circ b)^*$ is zero.
By equation \ref{Equ:delta}, the path $0\leadsto p^*\omega_L$ is induced by 
\[
\delta:=\sum_id_{dR}x_id_{dR}\tau_i\in(\wedge^{2}\Omega_{A(0)})^{-2}.
\]
On the other hand, by Equation \ref{Equ:dnu}, the path $0\leadsto q^*\omega_L$ is induced by 
\[\nu=-\sum_jd_{dR}y_jd_{dR}y_j.\]

By the concatenation of $\delta$, $u^*(\omega_L)$, $\nu^{-1}$, we see that 
\[
R(\omega_L)\sim\sum_id_{dR}x_id_{dR}\tau_i+\sum_jd_{dR}y_jd_{dR}y_j.
\]
Using the explicit description of $\omega$ in Example 5.10 of \cite{bbj}, we have $\kappa^*R(\omega_L)\sim\omega$.
\end{proof}

We end this section with a strengthening of some of the results. If one uses Example 5.12 from \cite{bbj} instead of Example 5.10, then \Cref{thm:generalmain} is true Zariski locally, instead of just \'etale locally. This strengthening has been pointed out to us in private communication by Hyeonjun Park, who is working on a similar result which is expected to appear soon. 
We are very grateful for several helpful comments and suggestions from Hyeonjun.

\begin{thm}\label{thm:mainzariski}
Let $\bf X$ be a $(-2)$-shifted symplectic scheme. Then, Zariski locally, $\bf X$ is equivalent to the derived intersection of two Lagrangians inside a $(-1)$-shifted symplectic derived scheme which is the shifted cotangent stack of a smooth classical scheme. 
\end{thm}
\begin{proof}
Since the proof is very similar to the proof in the \'etale case, we only point out the changes needed. We follow the notation in \cite[Example 5.12]{bbj}. The classical master equation is given by $\frac{1}{4}\sum_{j=1}^{m_1}\frac{1}{q_j}(\frac{\partial \Phi}{\partial y_j})^2=0$. Let $h_j=\frac{f_j}{2q_j}$ and $g_k^j=\frac{\partial f_j}{\partial x_k}-\frac{f_j}{2q_j}\frac{\partial q_j}{\partial x_k}$. Let $\nu=-\sum_j d_{dR}q_jy_jd_{dR}y_j$. Then the argument in this section goes through completely.
\end{proof}

\section{The derived moduli stack of coherent sheaves}
In this section, we work out an explicit example of the construction in Section 2.1, the derived moduli stack of length $d$ sheaves on $\bC^4$. We denote this moduli stack by $\mdc{4}$. For this section we assume that $k = \mb C$.

\subsection{Resolutions of the polynomial algebra}
In this subsection we recall some properties of polynomial algebras that will be useful when constructing cofibrant replacements later. While all results in this subsection are true for general commutative Koszul Calabi-Yau algebras, we restrict ourselves to the polynomial algebra case for concreteness. We work in the setup of \cite{operads}.

Let $V$ be an $n$-dimensional vector space over $k$ with basis $\{ x_1, \ldots, x_n\}$. The symmetric algebra $A := S(V) \simeq k[x_1, \ldots, x_n]$ is a quadratic algebra: it is a quotient of $T(V)$, the tensor algebra of $V$, by the quadratic ideal generated by elements of the form $x_i\otimes x_j - x_j \otimes x_i$. The exterior algebra $\wedge V$ is the Koszul dual coalgebra of $A$, and we will denote it by $\dual A$. We treat it as a nonpositively graded vector space i.e. $(\dual A)^{-k} = \wedge^k V$. The graded coalgebra structure is given by the shuffle coproduct:
\begin{align*}
    \Delta: \wedge^k V &\to \bigoplus_{i=0}^k \wedge^i V \otimes \wedge^{k-i}V\\
    x_1\wedge \ldots\wedge x_k &\mapsto \sum_{i=0}^k \sum_{(i+1, k-i)\text{-shuffles } \sigma} \varepsilon(\sigma) \left(x_{\sigma(0)} \wedge \ldots \wedge x_{\sigma(i)} \right) \otimes \left(x_{\sigma(p+1)} \wedge \ldots \wedge x_{\sigma(k)} \right)
\end{align*}
where we use the notation $x_0 = 1$ and $\varepsilon(\sigma)$ denotes the Koszul sign of $\sigma$. The natural inclusion $k \to \dual A$ is a coaugmentation map for this coalgebra and induces a decomposition $\dual A = k \oplus \overline{\wedge V}$, where $\overline{\wedge V} = \oplus_{k=1}^n \wedge^k V$. The coproduct $\Delta$ induces a reduced coproduct $\overline{\Delta}$ on $\overline{\wedge V}$, given by $\overline{\Delta}(\alpha) = \Delta(\alpha) - 1\otimes \alpha - \alpha \otimes 1$.

Denote by $\cobar A$ the free graded unital associative algebra on $\overline{\wedge V}[-1]$ (i.e. the graded tensor algebra) with product denoted by $\otimes$. We have, for example, that
\[ \left(\cobar A\right)^0 \simeq k\langle x_1, \ldots, x_n\rangle, \]
the free associative algebra on the generators of $V$, and $(\cobar A)^{-1}$ is generated as a $(\cobar A)^0$-module by the generators of $\wedge^2 V$. Note that the product $\otimes$ on $\cobar A$ should not be confused with the wedge product on $\wedge V$. The reduced coproduct $\overline\Delta$ can be extended to a graded derivation $d:\cobar A \to \cobar A$ of degree +1, making $\cobar A$ a differential graded algebra (dga). This is precisely the cobar construction of $\dual A$. The inclusion $V \to A$ (i.e. mapping $x_i$ to $x_i$) induces a dga map $p: \cobar A \to A$. By the Koszul criterion, we have
\begin{prop}[{e.g. \cite[Theorem 3.4.6]{operads}}]
\label{Prop:Koszulbimodres}
    The dga map $p: \cobar A \to A$ is a minimal resolution of $A$. In particular, it is a cofibrant replacement in the category of dgas.
\end{prop}
\subsection{Construction}
In this subsection we recall the construction of the derived moduli stack of coherent sheaves of length $d$ on $\C^n$, which is similar to the one in \cite{ricolfisavvas}. We continue using $A$ to denote the polynomial ring $k[x_1, \ldots, x_n]$. The moduli functor $\mnc$ assigns to a cdga $C$ the nerve of the category whose objects are coherent sheaves on $\BSpec C \times^h \C^n$ that are flat, 0-dimensional and of length $d$ over $\BSpec C$. Since such a sheaf is precisely a projective $C$-module of rank $n$ endowed with the structure of an $A$-module, we see that
\[ \mnc(C) \simeq Map_{dgCat}\left(A, Proj_d(C)\right) \]
where the mapping space is computed in the category $dgCat$ of dg-categories, and $Proj_d(C)$ is the dg-category of projective modules of rank $d$ over $C$. Since $A$ here is regarded as a dg-category with one object, a 0-simplex in this mapping space is a map of dgas $A\to End_C(P, P)$ for some $P \in Proj_d(C)$. Let $W$ be a vector space over $k$ of dimension $d$. Working \'{e}tale locally on $C$, we can assume that $P \simeq C\otimes_k W$ so that an endomorphism of $P$ is an endomorphism of $C\otimes_k W$ up to conjugation by $GL_d$.

Recall that $\cobar A$ denotes the cofibrant replacement of $A$ constructed earlier. The mapping space
\[ Map_{dga}\left(A, \End_C(C\otimes_k W)\right) \simeq Map_{dga}\left(A, \End_k(W)\otimes_k C\right) \]
is equivalent, by \cite[Theorem 4.2]{toen}, to the nerve of the category $\mc M(\cobar A, \End_k(W)\otimes_k C)$ whose objects are maps
\[ \cobar A \to Y \]
where $\End_k(W)\otimes_k C \to Y$ is a quasi-isomorphism of dgas, and morphisms are equivalences of these. Define, as in \cite{berest}, the cdga $A_d$ whose generators are $y_i^{jk}, 1 \leq j,k \leq d$ for each $y_i$ generating $\cobar A$, and differentials are those induced by $A$ (i.e. we treat each $y_i^{jk}$ as the $(j, k)^{th}$ entry of a matrix corresponding to $y_i$, and set the matrix differentials accordingly).

The map
\begin{align*}
    \phi: A &\to \End_k(W) \otimes_k A_d\\
    a &\mapsto \sum_{i, j} e_{ij}\otimes a_{ij}
\end{align*}
induces, by composition, a map, which we also call $\phi$, $\mc M(A_d, C) \to \mc M(\cobar A, \End_k(W)\otimes_k C)$. By \cite[Lemma 1]{berest}, this map is an equivalence. Combined with the equivalence $Map_{dga}(A_d, C) \simeq Map_{cdga}(A_d, C)$ (\cite[Theorem 2.12]{longknots}), we conclude that
\begin{equation}\label{eq:mapping}
    Map_{dga}\left(A, End_C(C\otimes_k W)\right) \simeq Map_{cdga}(A_d, C)
\end{equation}

\begin{defn}[\cite{berest}]
    The \textbf{derived representation scheme} parameterizing $d$-dimensional representations of the polynomial algebra $A$ is defined to be $\BSpec A_d$. We denote it by $\drep d$.
\end{defn}
Incorporating the conjugation action by $GL_d$, we have
\begin{lem}
    The map $\phi$ induces an equivalence of derived stacks
    \begin{align}\label{eq:atlas}
        \phi: \mf X_d :=  [\drep d/GL_d] \xrightarrow[]{\sim} \mnc.
    \end{align}
\end{lem}

\begin{ex}\label{eg:c4}
    We work out the details for the case $n=4$ since we will be using these later. Here $V$ is a four-dimensional vector space with basis $\{ x_1, \ldots, x_4 \}$, $A = k[x_1, \ldots, x_4]$, and $\cobar A$ is the semifree dg-algebra generated over $k$ by
    \begin{align*}
    x_i,\; & 1 \leq i \leq 4 &\text{ in degree $0$, with } & d(x_i) = 0\\
    c_{ij},\; & 1 \leq i < j \leq 4 &\text{ in degree $-1$, with } & d(c_{ij}) = [x_i, x_j]\\
    &s_{132}, s_{124}, s_{143}, s_{234} &\text{ in degree $-2$, with } & d(s_{ijk}) = [x_i,c_{jk}] + [x_j,c_{ki}] + [x_k,c_{ij}]\\
    &t &\text{ in degree $-3$, with } & d(t) = \Big(\Big.[x_1, s_{234}] + [x_2, s_{143}] + [x_3, s_{124}] + [x_4, s_{132}]\\
    &&& + [c_{12}, c_{34}] + [c_{13}, c_{42}] + [c_{14}, c_{23}] \Big.\Big),
    \end{align*}
    where we adopt the convention that $c_{lk} = - c_{kl}$ for $k, l \in \{ 1, \ldots, 4 \}$.
    
    Then $A_d$ is the semifree cdga generated over $k$ by the entries (denoted by the superscript $\mu\nu$ for $1 \leq \mu, \nu \leq d$) of the $d \times d$ matrices
    \begin{align*}
    X_i,\;  & 1 \leq i \leq 4 &\text{ in degree $0$, with } & d(X_i^{\mu\nu}) = 0\\
    C_{ij},\; & 1 \leq i < j \leq 4 &\text{ in degree $-1$, with } & d(C_{ij}^{\mu\nu}) = [X_i, X_j]^{\mu\nu}\\
    &S_{234}, S_{143}, S_{124}, S_{132} &\text{ in degree $-2$, with } & d(S_{ijk}^{\mu\nu}) = \left([X_i,C_{jk}] + [X_j,C_{ki}] + [X_k,C_{ij}]\right)^{\mu\nu}\\
    &T &\text{ in degree $-3$, with } & d(T^{\mu\nu}) = \Big(\Big. [X_1, S_{234}] + [X_2, S_{143}] + [X_3, S_{124}]\\
    &&& + [X_4, S_{132}] + [C_{12}, C_{34}] + [C_{13}, C_{42}] + [C_{14}, C_{23}] \Big.\Big)^{\mu\nu},
    \end{align*}
    where we once again adopt the convention that $C_{lk} = - C_{kl}$ for $k, l \in \{ 1, \ldots, 4 \}$.
\end{ex}

\subsection{Universal object}
We now give a description of the universal object for the stack $\mnc$.
\begin{defn}
    An object $\ms E \in QCoh(\mnc \times^h \mb C^n)$ is a \textbf{universal object} for the stack $\mnc$ if, for any map $x:\BSpec C \to \mnc$ from an affine derived scheme, corresponding to an object $F \in QCoh(\BSpec C \times^h \mb C^n)$, there is an equivalence $x^*\ms E \simeq F$.
\end{defn}

\begin{prop}[{\cite[Corollary 2.6]{bd19}}]
    There exists a universal object $\ms E$ for the stack $\mnc$.
\end{prop}

\begin{prop}[{\cite[Proposition 3.3]{bd19}}]
\label{Prop:tangentcomp}
    There is an isomorphism of Lie algebras
    \[ \mb T_{\mnc} \xrightarrow{\sim} \RHom_{\mnc}\left(\Upsilon\ms E, \Upsilon\ms E\right)[1],\]
    where $\Upsilon\ms E := \omega_{\mnc} \otimes\ms E$ for $\omega_{\mnc}$ the dualizing complex on $\mnc$.
\end{prop}

Let $h: \drep{d} \to \mf X_d$ denote the natural projection map. Then $\left(\phi\circ h\right)^*\ms E$ is an element of $QCoh(\drep{d}\times^h \mb C^n)$. Let $R_d := A_d\otimes_k A$ and denote by $F_d$ the $R_d$-module $A_d\otimes_k W$.

\begin{prop}
\label{Prop:Universalobj}
    There is an equivalence $\left(\phi\circ h\right)^*\ms E \simeq F_d$.
\end{prop}
\begin{proof}
    The proof is almost tautological. The map $h$ corresponds (up to conjugation by $GL_d$) to the identity map in $Map_{cdga}(A_d, A_d)$ by definition. This corresponds, under composition with $\phi$ and the equivalence \Cref{eq:mapping}, to the $A_d$-module $F_d$ with $A$-action given by $\phi$.
\end{proof}

\subsection{Pullback of the symplectic structure to an atlas}
We now restrict to the four-dimensional case, i.e. $n=4$ and we are interested in understanding the symplectic forms on $\mdc 4$. We will use notation consistent with \Cref{eg:c4}.

Recall the equivalence \Cref{eq:atlas} that realizes $\drep d$ as an atlas (in fact a \textit{standard form open neighborhood}, though not necessarily minimal, in the sense of \cite[Definition 2.7]{bbbbj}) for the stack $\mf X_d \simeq \mdc 4$. In this subsection we will compute the pullback of the canonical $(-2)$-shifted symplectic form $\omega_{\mdc 4}$ under the composition
\[ \drep d \xrightarrow{h} \mf X_d \xrightarrow{\phi} \mdc 4. \]
Since most of the content here is completely analogous to Sections 3 and 4 in \cite{katzshi}, we only sketch the results. 

We first give an identification of the following composition of quasi-isomorphisms using representatives:
\begin{equation}
\label{Equ:tancompiso}
h^*\bT_{\mathfrak{X}_d}\simeq (\phi\circ h)^*\bT_{\mathcal{M}^d(\mathbb{C}^4)}\simeq \RHom_{A_d}(F_d, F_d)[1],
\end{equation}
where the second quasi-isomorphism is obtained by combining Propositions \ref{Prop:tangentcomp} and \ref{Prop:Universalobj}.

Since $h^*\mathbb{L}_{\mathfrak{X}_d}=\Omega^1_{A_d}\rightarrow\mathfrak{g}^\vee\otimes A_d$, we have
\begin{equation*}
h^*\bT_{\mathfrak{X}_d}\simeq A_d^{d^2}\rightarrow A_d^{4d^2}\rightarrow A_d^{6d^2}\rightarrow A_d^{4d^2}\rightarrow A_d^{d^2}.
\end{equation*}
This is a complex of $A_d$-modules concentrated in degrees $-1$, $0$, $1$, $2$, $3$.

Next we consider $\RHom_{A_d}(F_d, F_d)$. Computing this will first require us to find a resolution for $F_d$. For this we will use results (and notation) from \cite{ginzburg}.
Consider the semifree dg-algebra $B$ which is the dg-subalgebra of $\Omega{\dual A}$ generated by $x_i$, $c_{ij}$ and $s_{ijk}$ with the same differentials.
The $B$-dg-module $\Omega^1 B$ is generated by the $d_{dR}x_i, d_{dR}c_{ij}, d_{dR}s_{ijk}$ while $\Der B$ is generated by $\partial/\partial x_i, \partial/\partial c_{ij}, \partial/\partial s_{ijk}$. Define the 2-form $\omega_B \in \Omega^2 B$ by
\begin{align*}
\omega_B &= d_{dR}x_1\otimes d_{dR}s_{234} + d_{dR}x_2\otimes d_{dR}s_{143} + d_{dR}x_3\otimes d_{dR}s_{124} + d_{dR}x_4\otimes d_{dR}s_{132}\\
& + d_{dR}c_{12}\otimes d_{dR}c_{34} + d_{dR}c_{13}\otimes d_{dR}c_{42} + d_{dR}c_{14}\otimes d_{dR}c_{23}.
\end{align*}
Then the triple $(B, \omega_B, d)$ is a symplectic data in the sense of \cite[\S 3.6]{ginzburg}, and the exact sequence (3.8.7) in op. cit. gives $A$-bimodule resolutions of $A$:
\begin{equation}
\label{bimodres}
\begin{tikzcd}
A\otimes A\arrow{r}\arrow[swap]{d} & \Der_0(B \vert A) \arrow{r}\arrow{d}{\iota_{\omega_B}} & \Der_{1}(B \vert A) \arrow{r}\arrow{d}{\iota_{\omega_B}} & \Der_{2}(B \vert A) \arrow{r}\arrow{d}{\iota_{\omega_B}} & A\otimes A \arrow{d}\\
A\otimes A\arrow{r} & \Omega_{-2}(B \vert A) \arrow{r} & \Omega_{-1}(B \vert A) \arrow{r}& \Omega_{0}(B \vert A) \arrow{r} & A\otimes A 
\end{tikzcd}
\end{equation}
where the upper row are viewed as a complex of $A\otimes A$-modules endowed with the ``inner'' bimodule structure i.e. $a(b'\otimes b'')c = b'c\otimes ab''$ (in contrast with the usual or ``outer'' bimodule structure $a(b'\otimes b'')c = ab'\otimes b''c$).

The isomorphism 
\begin{align*}
    \Der_{1}(B \vert A) &\simeq \Omega^1_{-1}(B \vert A)\\
    \frac{\partial}{\partial c_{ij}} &\mapsto d_{dR} c_{kl}, \qquad i, j \neq k, l
\end{align*}
is the one induced by contraction with $\omega_B$. 

We identify $V$ with the $k$-vector space with basis the (four) $B$-generators of $\Omega^1_0B$, and identify $\wedge^2 V$ with the $k$-vector space with basis the (six) $B$-generators of $\Omega^1_1 B$. Taking the first three terms in the upper row of (\ref{bimodres}), and the last three terms in the lower row of (\ref{bimodres}), we obtain a bimodule resolution of $A$ as
\begin{equation}\label{quiverres}
A\otimes A
\xrightarrow{\alpha^{-3}} A\otimes V^*\otimes A\xrightarrow{\alpha^{-2}}
A\otimes \wedge^2 V^* \otimes A \simeq A\otimes \wedge^2 V \otimes A
\xrightarrow{\alpha^{-1}}A\otimes V\otimes A\xrightarrow{\alpha^0} A\otimes A
\end{equation}
with the isomorphism $\wedge^2 V^* \xrightarrow{\sim} \wedge^2 V$ given by interior product with the volume form $x_1 \wedge x_2 \wedge x_3 \wedge x_4$ on $V$. This complex is analogous to the $3$-dimensional Calabi-Yau algebra case in \cite[Proposition 5.1.9]{ginzburg}.

Denoting $A_d\otimes A$ by $R_d$ and tensoring the complex \Cref{quiverres} with $R_d$ on the right over $A$ and $F_d$ on the left over $A$, we obtain an $R_d$-module resolution of $F_d$, which we denote by $F^\bullet$:
\begin{equation}
\label{Equ:ResF}
R_d\otimes F_d
\xrightarrow{\alpha^{-3}}R_d\otimes V^*\otimes F_d\xrightarrow{\alpha^{-2}}
R_d\otimes \wedge^2 V^* \otimes F_d \simeq F_d\otimes \wedge^2 V \otimes R_d
\xrightarrow{\alpha^{-1}}F_d\otimes V\otimes R_d\xrightarrow{\alpha^0} F_d\otimes R_d.
\end{equation}
Then $\RHom_{A_d}(F_d, F_d)$ is quasi-isomorphic to 
\begin{equation}
	\label{Equ:RHom}
	\begin{split}
		 F_d^*\otimes F_d\xrightarrow{\beta^{0}}
		 F_d^*\otimes V^*\otimes F_d\xrightarrow{\beta^{1}}
		 F_d^*\otimes \wedge^2V^*\otimes F_d\simeq  F_d\otimes \wedge^2V\otimes F_d^*\xrightarrow{\beta^{2}}F_d\otimes V\otimes F_d^*\xrightarrow{\beta^{3}}F_d\otimes F_d^*.
	\end{split}
\end{equation}
We denote this complex by $L^\bullet$. This is a complex of $A_d$ modules with generators in degrees $0$, $1$, $2$, $3$, $4$. 
Recall that $F_d=A_d\otimes_k W$. The following result is analogous to \cite[Lemma 8]{katzshi}:
\begin{prop}
\label{Prop:cotaniso}
    Let $\{f_i\}$ be a set of basis of $W$. Using the explicit representative $L^\bullet$ for $\RHom_{A_d}(F_d, F_d)$, the equivalence 
$   h^*\bT_{\mathfrak{X}_d}\simeq (\phi\circ h)^*\bT_{\mathcal{M}^d(\mathbb{C}^4)}\simeq \RHom_{A_d}(F_d, F_d)[1]$ 
in \Cref{Equ:tancompiso} is explicitly defined by the maps: 
\begin{equation}
\label{Equ:maptangentcomplex}
\begin{split}
\gamma^0: \hspace{5pt} 
& h^*\mb T_{\mathfrak{X}_d}^{0}\rightarrow  L[1]^{0},
\hspace{5pt}
\frac{\partial}{\partial X^{\mu\nu}_i}\mapsto f^*_\mu\otimes x^*_i\otimes f_\nu.\\
\gamma^{1}: \hspace{5pt}
& h^*\mb T_{\mathfrak{X}_d}^{1}\rightarrow  L[1]^{1}, \hspace{5pt}
 \frac{\partial}{\partial C_{ij}^{\mu\nu}}\mapsto f_\nu\otimes x_k\wedge x_l\otimes f_\mu^*.\\
\gamma^{2}: \hspace{5pt}
& h^*\mb T_{\mathfrak{X}_d}^{2}\rightarrow  L[1]^{2}, \hspace{5pt}
\frac{\partial}{\partial S_{ijk}^{\mu\nu}}\mapsto f_\nu\otimes x_{\{1234\}\setminus \{ijk\}}\otimes f_\mu^*.\\
\gamma^{3}: \hspace{5pt}
& h^*\mb T_{\mathfrak{X}_d}^{3}\rightarrow  L[1]^{3}, \hspace{5pt}
\frac{\partial}{\partial T^{\mu\nu}}\mapsto f_{\nu}\otimes f_\mu^*.\\
 \gamma^{-1}: \hspace{5pt}
 & h^*\mb T_{\mathfrak{X}_d}^{-1}\rightarrow  L[1]^{-1}, \hspace{5pt}
\frac{\partial}{\partial G^{\mu\nu}}\mapsto f_\mu^*\otimes f_\nu.
\end{split}
\end{equation}
Here $G^{\mu\nu}$'s are the generators of $h^*\mb T_{\mathfrak{X}_d}$ of degree $-1$. 
\end{prop}
\begin{proof}
The argument is the same as the proof of \cite[Lemma 8]{katzshi}.
Denote the underlying graded algebra/module of $R_d$ and $F_d$ by $\tilde{R}_d$ and $\tilde{F}_d$ respectively.  

Consider the the generators $\frac{\partial}{\partial X_i^{\mu\nu}}$, this term defines a first order deformation of $\tilde{F}_d$, i.e. a $\tilde{R_d}\otimes k[\epsilon]/(\epsilon^2)$ module $\mathcal{\tilde{F}}_d$ flat over $k[\epsilon]/(\epsilon^2)$, where $\epsilon$ is in degree $0$. This defines an element $\zeta_{X_i^{\mu\nu}}\in \Ext^1(\tilde{F}_d, \tilde{F}_d)$, and we have 
	\[\begin{tikzcd}
 ...\arrow{r} & \tilde{F}_d\otimes V\otimes \tilde{R}_d\arrow{r}\arrow[d, "e"] &\tilde{F}_d\otimes \tilde{R}_d\arrow{r}\arrow{d} & \tilde{F}_d\arrow{d}\arrow {r} & 0\\
 0\arrow{r} & \tilde{F}_d\arrow{r} &\mathcal{\tilde{F}}_d\arrow{r} &\tilde{F}_d\arrow{r} & 0
	\end{tikzcd}
	\]
 where the second row is given by $\zeta_{X_i^{\mu\nu}}$.
 Let $E_i^{\mu\nu}$ be the $d$ by $d$ matrix whose $(\mu, \nu)$ entry is $1$ and all other entries are $0$. Then the vertical map $e$ takes $f_{\mu}\otimes x_i\otimes 1$ to $f_\mu x_i-f_\mu(x_i+\epsilon E_i^{\mu\nu})$, and all other basis elements to zero. Thus 
 $\zeta_{X_i^{\mu\nu}}$ exactly corresponds to $f_\mu^*\otimes x_i^*\otimes f_\nu$ in the explicit representative \Cref{Equ:RHom} of $\RHom_{A_d}(F_d, F_d)$. 
This defines the map $\gamma^0:h^*\mb T_{\mathfrak{X}_d}^0\rightarrow L[1]^0$. The remaining maps can be defined analogously. 

Comparing the differentials we see that the maps $\gamma^\bullet$ define a morphism between the chain complexes $h^*\mb  T_{\mathfrak{X}_d}$ and $L^\bullet[1]$ which is the quasi-isomorphism in \Cref{Equ:tancompiso}.
\end{proof}

Define a $(-2)$-shifted 2-form $\omega = \left(\omega^0, 0, \ldots\right)$ on $A_d$ by
\begin{align*}
    \omega^0 = \tr \Big(\Big.
        & d_{dR} X_1 d_{dR}\, S_{234}
        + d_{dR} X_2 d_{dR}\, S_{143}
        + d_{dR} X_3 d_{dR}\, S_{124}
        + d_{dR} X_4 d_{dR}\, S_{132} \\
        & + d_{dR} C_{12} d_{dR}\, C_{34}
        + d_{dR} C_{13} d_{dR}\, C_{42}
        + d_{dR} C_{14} d_{dR}\, C_{23}
    \Big. \Big).
\end{align*}

\begin{thm}\label{thm:compareforms}
Under the above notation, $(\phi\circ h)^*\omega_{\boldsymbol{\mathcal{M}}^d(\C^4)}\simeq \omega$.
\end{thm}
\begin{proof}
By 
Proposition 3.3 and Proposition 5.3 in \cite{bd19}, the closed 2-form $\phi^*\omega_{\boldsymbol{\mathcal{M}}^d(\C^4)}$ is identified with the
Serre pairing
\begin{equation}
\label{Equ:Serrepairing}
\RHom_{\mf X_d}(F_d, F_d)^{\otimes 2}[n]\xrightarrow{\circ}\RHom_{\mf X_d}(F_d, F_d)[n]\xrightarrow{\text{tr}}\omega_{\mf X_d},
\end{equation}
where the last map is given by \cite[Equation (5.5)]{bd19}.
Using the explicit representative $L^\bullet$ of $\RHom_{A_d}(F_d, F_d)$, the pairing $(\phi\circ h)^*\omega_{\boldsymbol{\mathcal{M}}^d(\C^4)}$ is given by 
\[L^\bullet\otimes L^\bullet[n]\rightarrow L^\bullet[n]\rightarrow (L^\bullet)^\vee\xrightarrow{\text{tr}}A_d.
\]

We write an element $M\in\Hom(F^\bullet, F^\bullet[1])$ as $M=[M_1, M_2, M_3, M_4]^T$, where $M_i$ is the matrix with respect to the basis $f_\mu^*\otimes x_i^*\otimes f_\nu$. We write an element $N\in \Hom(F^\bullet, F^\bullet[3])$ as $N=[N_1, N_2, N_3, N_4]$ where $N_j$ is the matrix with respect to the basis $f_\nu\otimes x_j\otimes f_\mu^*$. We write an element $C\in\Hom(F^\bullet, F^\bullet[2])$ as $C=[C_{12}, C_{13}, C_{14}, C_{23}, C_{24}, C_{34}]$, where $C_{kl}$ is the matrix with respect to the basis $f_\mu\otimes x_k\wedge x_l\otimes f_\nu^*$. 
Then the Serre pairing in \Cref{Equ:Serrepairing} between $\Hom(F^\bullet, F^\bullet[1])$ and $\Hom(F^\bullet, F^\bullet[3])$ is given by $\text{tr}(N_i\circ M_i)$. Via the isomorphism $\gamma^i$ in Proposition \ref{Prop:cotaniso}, this pairing corresponds to the term $\tr(\dd X_i\dd S_{\{1234\}\setminus i})$. The Serre pairing between $\Hom(F^\bullet, F^\bullet[2])$ and $\Hom(F^\bullet, F^\bullet[2])$ is given by $\text{tr}(C_{ij}\circ C_{kl})$. Via the isomorphism $\gamma^i$, this pairing corresponds to the term $\tr(\dd C_{ij}\dd C_{kl})$.  
Or equivalently, the pairing is just given by the isomorphism $L^\bullet[n]\simeq L^\vee$.
From the above discussion, the pairing in \Cref{Equ:Serrepairing} is the same as $\omega$.
\end{proof}

\subsection{The moduli of sheaves on four-dimensional affine space as a Lagrangian intersection}
In this section we obtain the moduli of 0-dimensional sheaves on $\mb C^4$ as a Lagrangian intersection. First we recall (partially) a result of \cite[Theorem 2.10]{bbbbj}:
\begin{thm}[{\cite[Theorem 2.10]{bbbbj}}]\label{thm:bbbbj}
    \begin{enumerate}
        \item Let $({\bf X}, \omega_{\bf X})$ be a $(-2)$-shifted derived Artin stack, and $x \in {\bf X}$. Then there exists a standard form cdga $A$ which is minimal at $p \in \Spec H^0(A)$, a $(-2)$-shifted symplectic form on $\BSpec A$, and a smooth morphism ${\bf f} : {\bf U} = \BSpec A \to {\bf X}$ of relative dimension $n$ with ${\bf f} (p) = x$ and ${\bf f}^*\omega_{\bf X} \sim \omega$. Furthermore, up to taking a finite \'etale cover, $A$ and $\omega$ are in modified strong Darboux form as follows:
    \begin{enumerate}
        \item The degree 0 part $A^0$ of $A$ is a smooth $k$-algebra of dimension $m_0 = \dim H^0\left( \mb L_{\bf X}|_p\right)$, and we are given $x_1, \ldots, x_{m_0} \in A^0$ such that $d_{dR}x_1, \ldots, d_{dR}x_{m_0}$ form a basis for $\Omega^1_{A^0}$ over $A^0$.
        \item As a graded commutative algebra, $A$ is generated over $A^0$ by
        \begin{align*}
            z_1, \ldots, z_{m_1} &\text{ in degree } -1,\\
            y_1, \ldots, y_{m_0} &\text{ in degree } -2,\\
            w_1, \ldots, w_n &\text{ in degree } -3
        \end{align*}
        for $m_1 = \dim H^{-1}\left( \mb L_{\bf X}|_p\right)$ and $n = \dim H^1\left( \mb L_{\bf X}|_p\right)$, the relative dimension of ${\bf f}$. The differential is defined using a Hamiltonian $H \in A^{-1}$ in the usual manner as in \cite{bbj}. The symplectic form $\omega = (\omega^0, 0, \ldots)$ with
        \[ \omega^0 = \sum_{i=1}^{m_0} d_{dR} y_i d_{dR} x_i + \sum_{j=1}^{m_1} d_{dR}z_j d_{dR}z_j. \]
    \end{enumerate}
    \item Let $B$ be the subalgebra of $A$ generated by $A^0$ and the variables $y_i, z_j$ for all $i, j$ with inclusion $\iota: B \to A$. Then $B$ is closed under $d$ and so is a dg-subalgebra of $A$. Now $H \in B^{-1} = A^{-1}$, and $\omega$ is the image of a $(-2)$-shifted symplectic structure $\omega_B$ on $B$. Moreover $B$ and $\omega_B$ are in Darboux form, and $B$ is minimal at $p$. Geometrically, we have a diagram of morphisms
    \[\begin{tikzcd}
	{{\bf V} = \BSpec B} && {{\bf U} = \BSpec A} & {{\bf X}}
	\arrow["{{\bf f}}", from=1-3, to=1-4]
	\arrow["{{\bf i} = \BSpec \iota}"', from=1-3, to=1-1]
    \end{tikzcd}\]
    where $({\bf X}, \omega_{\bf X})$ and $({\bf V}, \omega_B)$ are $(-2)$-shifted symplectic with ${\bf f}^*(\omega_{\bf X}) \sim {\bf i}^*(\omega_B)$ in $(-2)$-shifted closed 2-forms on ${\bf U}$. We can think of $\bf f$ as a ``submersion'' and $\bf i$ as an embedding of ${\bf U}$ as a ``derived subscheme'' of $\bf V$. On classical truncations, $i = t_0({\bf i}): U = t_0({\bf U}) \to t_0({\bf V}) = V$ is an isomorphism. There is an equivalence of relative (co)tangent complexes:
    \begin{equation}
        \mb L_{{\bf U}/{\bf V}} \simeq \mb T_{{\bf U}/{\bf X}}[3].
    \end{equation}
    \end{enumerate}
\end{thm}
Our aim in this section is to obtain analogs of $A$ and $B$ in the case where ${\bf X}$ is the derived moduli stack of 0-dimensional length $d$ sheaves on $\mb C^4$, and to show that $B$ is a (nice) derived critical locus and, in particular, a Lagrangian intersection. Note that our $A$ and $B$ will not necessarily be minimal as guaranteed by the theorem, but by sacrificing minimality we can obtain a much nicer description.

Recall the notation from \Cref{eg:c4} and that the shifted form $h^*\omega_{\mnc} = \left(\omega^0, 0, \ldots\right)$ is the 2-form on $A_d$, by \Cref{thm:compareforms}, where
\begin{align*}
    \omega^0 = \tr \Big(\Big.
        & d_{dR} X_1 d_{dR}\, S_{234}
        + d_{dR} X_2 d_{dR}\, S_{143}
        + d_{dR} X_3 d_{dR}\, S_{124}
        + d_{dR} X_4 d_{dR}\, S_{132} \\
        & + d_{dR} C_{12} d_{dR}\, C_{34}
        + d_{dR} C_{13} d_{dR}\, C_{42}
        + d_{dR} C_{14} d_{dR}\, C_{23}
    \Big. \Big).
\end{align*}
Now let $B_d$ be the subalgebra of $A_d$ generated freely over $A_d^0$ by the $C_{ij}$'s and the $S_{ijk}$'s, with the same differentials. Then there is a natural inclusion $\iota: B_d \to A_d$, with $\omega_{B_d} := \iota_* (\phi\circ h)^*\omega_{\mnc}$ also of the form $\omega_{B_d} = \left(\omega_{B_d}^0, 0, \ldots\right)$, where
\begin{align*}
    \omega_{B_d}^0 = \tr \Big(\Big.
        & d_{dR} X_1 d_{dR}\, S_{234}
        + d_{dR} X_2 d_{dR}\, S_{143}
        + d_{dR} X_3 d_{dR}\, S_{124}
        + d_{dR} X_4 d_{dR}\, S_{132} \\
        & + d_{dR} C_{12} d_{dR}\, C_{34}
        + d_{dR} C_{13} d_{dR}\, C_{42}
        + d_{dR} C_{14} d_{dR}\, C_{23}
    \Big. \Big).
\end{align*}
The desired classical master equation here is given by
\begin{equation}\label{eq:classicalmaster}
    \tr \left( [X_1, X_2][X_3, X_4] + [X_1, X_3][X_4, X_2] + [X_1, X_4][X_2, X_3] \right) = 0
\end{equation}

\begin{prop}
    The pair $(B_d, \omega_{B_d})$ are in Darboux form. In particular $(\BSpec B_d, \omega_{B_d})$ is a $(-2)$-shifted symplectic derived scheme.
\end{prop}
\begin{proof}
    It suffices to show that $(B_d, \omega_{B_d})$ are constructed in a manner similar to \cite[Example 5.10, Remark 5.11]{bbj}. The $k$-algebra $B_d^0 = k[X_1^{\mu\nu}, \ldots, X_4^{\mu\nu}]$ is clearly smooth of dimension $m_0 = 4d^2$ with the the $X_i^{\mu\nu}$ forming a basis for $\Omega^1_{B_d^0}$ over $B_d^0$. The $C_{ij}^{\mu\nu}$'s form $m1 = 6d^2$ generators in degree $-1$, while the $S_{ijk}^{\mu\nu}$'s form $m_0 = 4d^2$ generators in degree $-2$. Observe that $m_1$ is even so we are in the situation of \cite[Remark 5.11]{bbj}. Now define the Hamiltonian $\Phi \in B_d^{-1}$ by
    \begin{align*}
    \Phi := \tr \Big(\Big.
        & C_{12}[X_3, X_4] + C_{13}[X_4, X_2] + C_{14}[X_2, X_3]\\
        & + C_{34}[X_1, X_2] + C_{42}[X_1, X_3] + C_{23}[X_1, X_4]
    \Big. \Big).
    \end{align*}
    Then we have
    \begin{align*}
        \frac{\partial \Phi}{\partial C_{12}^{\mu\nu}} = \left([X_3, X_4]^T\right)^{\mu\nu} = d\left(C_{34}^T\right)^{\mu\nu},\qquad & \frac{\partial \Phi}{\partial \left(C_{34}^T\right)^{\mu\nu}} = \left([X_1, X_2]\right)^{\mu\nu} = d(C_{12})^{\mu\nu},\\
        \frac{\partial \Phi}{\partial C_{13}^{\mu\nu}} = \left([X_4, X_2]^T\right)^{\mu\nu} = d\left(C_{42}^T\right)^{\mu\nu},\qquad & \frac{\partial \Phi}{\partial \left(C_{42}^T\right)^{\mu\nu}} = \left([X_1, X_3]\right)^{\mu\nu} = d(C_{13})^{\mu\nu},\\
        \frac{\partial \Phi}{\partial C_{14}^{\mu\nu}} = \left([X_2, X_3]^T\right)^{\mu\nu} = d\left(C_{23}^T\right)^{\mu\nu},\qquad & \frac{\partial \Phi}{\partial \left(C_{23}^T\right)^{\mu\nu}} = \left([X_1, X_4]\right)^{\mu\nu} = d(C_{14})^{\mu\nu},
    \end{align*}
    \begin{align*}
        \frac{\partial \Phi}{\partial X_1^{\mu\nu}} = \left( [X_2, C_{34}]^T + [X_3, C_{42}]^T + [X_4, C_{23}]^T \right)^{\mu\nu} &= d(S_{234}^T)^{\mu\nu},\\
        \frac{\partial \Phi}{\partial X_2^{\mu\nu}} = \left( [X_4, C_{31}]^T + [X_3, C_{14}]^T + [X_1, C_{43}]^T \right)^{\mu\nu} &= d(S_{143}^T)^{\mu\nu},\\
        \frac{\partial \Phi}{\partial X_3^{\mu\nu}} = \left( [X_4, C_{12}]^T + [X_2, C_{41}]^T + [X_1, C_{24}]^T \right)^{\mu\nu} &= d(S_{124}^T)^{\mu\nu},\\
        \frac{\partial \Phi}{\partial X_4^{\mu\nu}} = \left( [X_3, C_{21}]^T + [X_2, C_{13}]^T + [X_1, C_{32}]^T \right)^{\mu\nu} &= d(S_{132}^T)^{\mu\nu}, \text{ and }\\
        \frac{\partial \Phi}{\partial S_{ijk}^{\mu\nu}} = 0 = d(X_l^T)^{\mu\nu} &\text{ for all } i, j, k, l.
    \end{align*}
    It follows that the Hamiltonian $\Phi$ satisfies the equation
    \begin{equation}
        \sum_{\mu, \nu} \left(\frac{\partial \Phi}{\partial C_{12}^{\mu\nu}} \frac{\partial \Phi}{\partial \left(C_{34}^T\right)^{\mu\nu}}
        + \frac{\partial \Phi}{\partial C_{13}^{\mu\nu}} \frac{\partial \Phi}{\partial \left(C_{42}^T\right)^{\mu\nu}}
        + \frac{\partial \Phi}{\partial C_{14}^{\mu\nu}} \frac{\partial \Phi}{\partial \left(C_{23}^T\right)^{\mu\nu}}\right)
        = 0
    \end{equation}
    which is precisely the classical master equation, \Cref{eq:classicalmaster}, as desired. Now define the $(-2)$-shifted 1-form $\phi$ on $B_d$ by
    \begin{align*}
        \phi = \tr \Bigg(\Bigg.
        & 2\left( S_{234}\left(d_{dR}X_1\right) + S_{143}\left(d_{dR}X_2\right) + S_{124}\left(d_{dR}X_3\right) + S_{132}\left(d_{dR}X_4\right)\right)\\
        & - \Big(\Big. C_{12}\left(d_{dR}C_{34}\right) + C_{13}\left(d_{dR}C_{42}\right) + C_{14}\left(d_{dR}C_{23}\right)\\
        & + C_{34}\left(d_{dR}C_{12}\right) + C_{42}\left(d_{dR}C_{13}\right) + C_{23}\left(d_{dR}C_{14}\right)\Big.\Big)
        \Bigg.\Bigg).
    \end{align*}
    Then one checks that $d\phi = -d_{dR}\Phi$ and $d_{dR}\phi = -2\omega_{B_d}^0$, and we're done.
\end{proof}

We now come to the main result of this section, which is a straightforward application of the above result and \Cref{prop:evencase}.
\begin{thm}\label{thm:maindim4}
    The $(-2)$-shifted symplectic derived scheme $(\BSpec B_d, \omega_{B_d})$ is equivalent to a derived critical locus.
\end{thm}
So we get the following picture
\[\begin{tikzcd}
	{{\bf V}_d = \BSpec B_d} && {{\drep{d}} = \BSpec A_d} & {{\mdc{4}}}
	\arrow["{\phi\circ h}", from=1-3, to=1-4]
	\arrow["{{\bf i} = \BSpec \iota}"', from=1-3, to=1-1]
\end{tikzcd}\]
analogous to the one in \cref{thm:bbbbj}(2) (except for, as stated above, the minimality properties).

We end this section with a brief description of how this construction can be motivated starting from a purely classical setting. We will rewrite $\drep{d, 4} := \drep{d}$ and ${\bf V}(d, 4) := {\bf V}_d$ to indicate that we are working over $\mb C^4$. Generalizing this, one can check that we get a similar picture for every $n \geq 3$:
\[\begin{tikzcd}
	{{\bf V}(d, n)} && {\drep{d, n}} & {{\mnc}}
	\arrow["{\phi\circ h}", from=1-3, to=1-4]
	\arrow["{{\bf i}}"', from=1-3, to=1-1]
\end{tikzcd}\]
with $t_0({\bf V}(d, n)) = t_0(\drep{d, n})$ being identified with the space of pairwise commuting $n$-tuples of $d\times d$ matrices. The picture exists for $n=2$ as well, except that in this case, while $t_0(\drep{d, 2})$ is still identified with the space of pairs of commuting $d\times d$ matrices, $t_0({\bf V}(d, 2)) = \mb C^{2d^2}$ is the (larger) space of pairs of $d\times d$ matrices.

We now have the following relationship between ${\bf V}(d, 2)$ and ${\bf V}(d, 3)$: ${\bf V}(d, 2) \times \mb C^{d^2} = \{ (X, Y, Z) | X, Y, Z \in \End_{\mb C}(\mb C^d)\}$ is the space of triples of $d\times d$ matrices, and ${\bf V}(d, 3)$ is obtained as the \textit{derived} critical locus of the function $\tr(X[Y, Z])$ defined on ${\bf V}(d, 2) \times \mb C^{d^2}$.

Next, working through the specifics of \Cref{prop:evencase} that give us the above theorem, we see that ${\bf V}(d, 4)$ is obtained as the \textit{derived} critical locus of a \textit{shifted} function $\Psi$ defined on ${\bf V}(d, 3) \times \mb C^{d^2}$. Using the letter $W$ to denote the matrix generator of the additional factor in ${\bf V}(d, 3) \times \mb C^{d^2}$, and the letters $C_X, C_Y$ and $C_Z$ to denote the degree $(-1)$ generators of ${\bf V}(d, 3)$ which map respectively to the commutators $[Y, Z], [Z, X]$ and $[X, Y]$, the shifted function is
\[ \Psi = [W, X]C_X + [W, Y]C_Y + [W, Z]C_Z. \]

More generally, ${\bf V}(d, n)$ for $n \geq 3$ is always obtained as the derived critical locus of a $(3-n)$-shifted function on ${\bf V}(d, n-1) \times \mb C^{d^2}$. This gives us an iterated construction of these moduli that begins with a purely classical space and naturally involves derived moduli spaces. If, as outlined in \Cref{sec:generalizations}, one works with the derived stacks instead of their atlases, we can dispense with the ${\bf V}(d, n)$'s and obtain a similar result for the $\mdc{n}$'s themselves i.e. that $\mdc{3} \to \mdc{2}\times \mb C^{d^2}$ is a derived critical locus, $\mdc{4} \to \mdc{3}\times \mb C^{d^2}$ is a derived critical locus, and so on.

\printbibliography
\end{document}